\documentclass{article}
\usepackage[utf8]{inputenc}
\usepackage{subfigure}
\usepackage{amsmath}
\usepackage{amsthm}
\usepackage{amssymb}
\usepackage{tablefootnote}
\usepackage{hyperref}
\usepackage[ruled,vlined]{algorithm2e}
\newtheorem{prop}{Proposition}[section]

\newtheorem{thm}[prop]{Theorem}

\theoremstyle{definition}

\theoremstyle{definition}
\newtheorem{eg}[prop]{Example}

 \def \tss {\textsuperscript}    \def \s {\sigma}            \def \Av {\text{Av}} \def \es {\emptyset}

\author{Yonah Biers-Ariel}
\title{Flexible Schemes for Pattern-Avoiding Permutations}

\begin{document}\sloppy
\maketitle

\begin{abstract}
We modify the enumeration schemes of Zeilberger and Vatter so that they can efficiently enumerate many new classes of pattern-avoiding permutations including all such classes with a regular insertion encoding.
\end{abstract}

\section{Introduction} \label{intro}
\subsection{Automatic Enumeration of Permutation Patterns}
The study of pattern-avoiding permutations dates back more than a century to MacMahon \cite{MacMahon}, but expanded significantly after Knuth addressed it in \textit{The Art of Computer Programming} \cite{Knuth}. In the last twenty years, an increasing amount of research has focused on the automatic enumeration of permutation classes. Rather than the traditional approach of considering a single avoidance class (or family of avoidance classes), these algorithms seek to input any set of patterns and then (hopefully) return some certificate which gives a polynomial-time algorithm to generate the enumeration sequence of the permutations which avoid these patterns.

The first of these were enumeration schemes, developed by Zeilberger in \cite{ZSchemes} and expanded by Vatter in \cite{VSchemes}. The idea is to separate permutations into groups by their prefixes and then attempt to reduce these classes to simpler groups by deleting prefix elements. These schemes were extended to words by Pudwell in \cite{Pudwell} and to dashed patterns by Baxter and Pudwell in \cite{Baxter}, although this paper is concerned only with the traditional permutation pattern case. Vatter's enumeration schemes are described in greater detail in section \ref{schemes}.

The other enumeration algorithm which we are interested in is the insertion encoding developed by Albert, Linton, and Ru\v{s}kuc in \cite{Ins}. The strategy here is to consider a finite automaton which moves between states, each state specifying which elements are already included in the permutation and where new elements will be inserted. It is described in greater detail in section \ref{insenc}.

Both of these algorithms, along with the new enumeration scheme that we propose in this paper, follow the paradigm of building a permutation element by element and finding rules to reduce partially-completed permutations to simpler ones. Another paradigm is illustrated by Bean, Gudmundsson, and Ulfarsson in \cite{struct}, where they decompose pattern-avoiding permutations into component structures which can more easily be counted. Here, though, we are only concerned with the first paradigm.

The object of this paper is to partially answer a question posed by Vatter in \cite{VSchemes}: Is there an automatic enumeration algorithm that applies to all permutation classes with finite enumeration schemes, all classes with regular insertion encodings, and all classes with only finitely many simple permutations? Here we provide an algorithm to produce what we will call flexible schemes. Finite flexible schemes exist for every class with a finite enumeration scheme or a regular insertion encoding (and many classes with neither). Like traditional schemes, flexible schemes do not provide generating functions -- indeed, they exist for avoidance classes which are believed not to have any D-finite generating function \cite{no_gf} -- but they do provide polynomial-time enumeration of their permutation classes. 

In the remainder of Section \ref{intro}, we present traditional enumeration schemes and the insertion encoding. In Section \ref{flexible}, we describe how flexible schemes differ from traditional ones, and present an automatic way to find them. In Section \ref{conditions} we show that a finite flexible scheme exists whenever either a finite traditional scheme or regular insertion encoding does. Finally, in section \ref{results}, we present some of the successes we have had on specific avoidance classes.

\subsection{Definitions}
In general, we use standard definitions for permutation patterns and related objects. A \emph{permutation} is some reordering of the numbers in $[n] = \{1,2,\dots,n\}$. Permutations are either written as a string of ordered values (like 24513), and the empty permutation is written as $\es$. A string $s$ of $k$ unique elements of $[n]$ can be turned into a permutation by taking its \emph{reduction}, the unique permutation of length $k$ whose elements occur in the same order as $s$.

A permutation $\s = \s_1\s_2 \dots \s_n$ \emph{contains} a pattern $p = p_1p_2 \dots p_k$ if there exists a sequence $i_1, i_2, \dots, i_k$ such that $\s_{i_1}\s_{i_2}\dots \s_{i_k}$ reduces to $p$. In this case, we call $\s_{i_1}\s_{i_2}\dots \s_{i_k}$ an \emph{occurrence} of $p$ in $\s$; if there is no occurrence of $p$ in $\s$, then $\s$ \emph{avoids} $p$. The following example illustrates these definitions.

\begin{eg}
The permutation 24513 contains the pattern 132 because 253 reduces to 132 (and hence is an occurrence of 132). However, 24513 avoids 321 because it does not contain three elements in decreasing order.
\end{eg}

When discussing Zeilberger's enumeration schemes, we need to categorize permutations by their prefixes. A \emph{prefix} of a permutation is the reduction of one of the permutation's initial segments. When we turn to Vatter's enumeration schemes and then the new ones introduced in this paper, we will instead categorize permutations by their downfixes. A \emph{downfix} of a permutation consists of the permutation's elements which lie below a particular value kept in their proper order.

\begin{eg}
The permutation 24513 has six prefixes: $\es, 1, 12, 123, 2341,$ and $24513$. It also has six downfixes: $\es, 1, 21, 213, 2413,$ and $24513$.
\end{eg}

When discussing Vatter's enumerations schemes, we will also need to categorize permutations by their \emph{gap vectors}. A permutation $\s_1\s_2\dots \s_n$ with downfix $\s_{i_1}\s_{i_2} \dots \s_{i_l}$ has corresponding gap vector $[i_1 - 1, i_2 - i_1 -1, \dots, i_l-i_{l-1}-1, n - i_l]$. We are particularly interested in gap vectors whose components are larger than those of some other vector, and so we say a gap vector $\mathbf{g} = [g_1,g_2, \dots, g_{l+1}]$ \emph{satisfies} a gap condition $\mathbf{h} = [h_1,h_2,\dots, h_{l+1}]$ if $g_i \ge h_i$ for all $1 \le i \le l+1$. When $\mathbf{g}$ satisfies $\mathbf{h}$, we write $\mathbf{g} \succeq\mathbf{h}$.

\begin{eg}
Consider the permutation 24513 and its downfix 21. It has corresponding gap vector $[0,2,1]$ which satisfies $[0,1,1]$ and $[0,2,0]$, but not $[1,1,1]$. Notice if a permutation has length $n$ and its downfix has length $l$, then the corresponding gap vector must have length $l+1$, and its elements must sum to $n-l$.
\end{eg}

We will often want to talk about the set of permutations avoiding some set of patterns $B$ with downfix $\pi$ and corresponding gap vector $\mathbf{g}$. We call this set $Z(B, \pi, \mathbf{g})$. When this set is nonempty, we say that $\mathbf{g}$ is \emph{viable} for $B$ and $\pi$ and when it is empty we say that $\mathbf{g}$ is \emph{nonviable} for $B$ and $\pi$. When $B$ and $\pi$ are clear from the context, we simply say that $\mathbf{g}$ is viable or nonviable.

Sometimes, we also want to talk about the set of permutations with downfix $\pi$ and corresponding gap vector $\mathbf{g}$ without worrying about what patterns they avoid. We call this set $Y(\pi, \mathbf{g})$.

\begin{eg}
The gap vector $[2,1,0]$ is viable for pattern set $\{123\}$ and downfix 12 because 
\[Z(\{123\}, 12, [2,1,0]) = \{35142, 43152, 45132, 53142, 54132\}.\]
On the other hand, the gap vector $[0,0,1]$ is nonviable. When we don't worry about avoiding 123, we find that 
\[Y(12, [2,1,0]) = \{34152, 35142, 43152, 45132, 53142, 54132\}.\]
\end{eg}

The key operation at the heart of enumeration schemes is that of deleting a downfix element to yield a simpler permutation. When a permutation with downfix $\pi$ and gap vector $\mathbf{g}$ has the $i\tss{th}$ element of its downfix deleted, it yields a permutation with downfix $d_i(\pi)$ and gap vector $d_i(\mathbf{g})$. Precisely, we define $d_i(\pi_1\pi_2 \dots \pi_k)$ to be the reduction of $\pi_1\dots \pi_{i-1} \pi_{i+1}\dots \pi_k$, and $d_i([g_1,g_2,\dots,g_{k+1}])$ $= [g_1, \dots, g_{i-1}, g_i + g_{i+1}, g_{i+2}, \dots, g_{k+1}]$. 

We also need to consider the \emph{refinements} of $\pi$, that is all the permutations formed by inserting the element $|\pi| +1$ somewhere in $\pi$. When we refine $\pi$, we also need to change $[g_1,g_2,\dots g_{k+1}]$ by deleting 1 from $g_i$ (representing the element which is now in the downfix), inserting a new element $g'$ between $g_i$ and $g_{i+1}$, and splitting the mass of $g_i$ between it and $g'$ (representing the elements that were between the $\pi_i$ and $\pi_{i+1}$, but are now between $\pi_{i+1}$ and $\pi_{i+2}$). Precisely, we let 
\[f_i(\pi_1\pi_2\dots \pi_k) = \pi_1\dots \pi_i (k+1) \pi_{i+1} \dots \pi_k, \text{ and}\] 
\[f_{i,j}([g_1,g_2,\dots, g_{k+1}]) = [g_1,\dots, g_{i-1}, j, g_i - j -1, g_{i+1}, \dots, g_{k+1}].\]

\begin{eg}
Let $\pi = 24513$ and $\mathbf{g} = [1,2,1,2,1,2]$. Then $d_1(\pi) = 3412$ and $d_1(\mathbf{g}) = [3,1,2,1,2]$. Similarly, $f_2(\pi) = 264513$ and $f_{2,1}(\mathbf{g}) = [1,1,0,1,2,1,2]$. 
\end{eg}

\subsection{Enumeration Schemes} \label{schemes}
The history and basic idea of enumeration schemes have already been discussed; here we illustrate the schemes with an extended example due to Vatter. 
\begin{eg}
Suppose we want an algorithm to enumerate the permutations which avoid the two patterns 1342 and 1432. We begin by considering all permutations with the downfix 1 (i.e. all permutations). Our goal is to find elements of the downfix such that if an occurrence of a forbidden pattern uses the element, then there is a different occurrence of a forbidden pattern which does not use the element. When we find such an element, we can delete it, secure in the knowledge that the resulting permutation contains a forbidden pattern if and only if the original one did. These elements are called \emph{reversely deletable} (in \cite{ZSchemes}) or \emph{ES-reducible} (in \cite{VSchemes}). We follow the more recent source and call them ES-reducible. When a downfix has an ES-reducible element, we will also call the downfix itself ES-reducible. When a downfix is not ES-reducible, we say that it is \emph{ES-irreducible}.

One can easily find permutations with the downfix 1 which contain either the pattern 1342 or 1432 but which contain neither pattern when the 1 is removed (for example, 1342 itself is such a permutation). Therefore, we consider the \emph{refinements} of 1, that is all the length-2 permutations for which 1 is a downfix. These refinements are 12 and 21. Looking at 21, we see that the second element is ES-reducible; if any permutation uses the 1 in an occurrance of 1342 or 1432, the pattern could just as easily be formed using the 2 instead. Turning to 12, though, we find that neither element is ES-reducible; indeed we can form forbidden subpatterns which use both of them.

We are saved, though, by considering gap vectors. If a permutation has at least two elements between the 1 and the 2 (i.e. if its gap vector corresponding to 12 satisfies $[0,2,0]$) , then either a 1342 or 1432 pattern must occur. On, the other hand, if it has fewer than two elements between the 1 and the 2 (i.e. if the gap vector fails to satisfy $[0,2,0]$), then 2 is ES-reducible. Therefore, we either replace 12 with a shorter downfix, or else we can ignore it entirely.

In Section \ref{details}, we explain precisely how a computer could record these rules, and how it would use them to generate terms of the permutation class' enumeration sequence.
\end{eg}

\subsection{Insertion Encoding} \label{insenc}
The insertion encoding encodes a process in which a permutation is built up by inserting its elements from smallest to largest, but only in designated slots. These slots, designated by $\diamond$, are the only places in which a new element can be added, and they must end up containing an element. Each permutation is constructed by a unique sequence of insertions. The following example shows how 24513 is constructed.
\begin{eg}
\[ \diamond\]
\[\diamond 1 \diamond\]
\[2\diamond 1 \diamond \]
\[2\diamond 13\]
\[24\diamond 13\]
\[24513\]
\end{eg}

Albert et al. in \cite{Ins} describe this sequence by recording at each step which slot an element is inserted into, numbering them beginning with 1 on the left. They also record where in its slot each new element is added. An $m$ represents inserting an element in the middle of a slot (leaving $\diamond$s on both sides), an $r$ represents inserting it on the right of a slot (leaving a $\diamond$ on its left), an $l$ represents inserting it on the left of a slot (leaving a $\diamond$ on the right), and an $f$ represents filling the slot (leaving no $\diamond$ at all). So, 24513 is recorded as $m_1l_1f_2l_1f_1$. 

When the strings that represent valid permutations in the avoidance class form a regular language, we say the insertion encoding is regular, and the avoidance class can be efficiently enumerated. The authors were able to precisely characterize classes with regular insertion encodings: they are the classes which contain only finitely many \emph{vertical alternations}, i.e. permutations where each odd element is greater than each even element or vice-versa. Later, Vatter showed in \cite{VIns} that the these classes could also be characterized as those for which any sufficiently long partial permutation has an insertion encoding reducible element, i.e. an element which can be removed without affecting the set of insertion sequences that could finish the permutation.
 
 \section{Flexible Schemes} \label{flexible}
 \subsection{Motivation}
 In this section, we introduce a new idea to extend traditional enumeration schemes and enable them to count many more avoidance classes. It is motivated by the following question: What if there is a downfix and gap condition which do not guarantee that all permutations with that downfix and satisfying that condition contain a forbidden pattern, but which do allow some element of the downfix to be deleted?
  
If a downfix has such an element for every possible gap vector, then we always are able to reduce it to a simpler downfix. Of course, this is not very useful if it is hard to tell, for a given gap vector, which element is the reducible one. However, if a downfix contains such an element for every possible gap vector, and if we can determine which element that is by comparing the gap vector to finitely many gap conditions, then we can efficiently reduce the downfix. Because we allow the element being reduced to change based on the gap vector, we call such a downfix \emph{Flexible Scheme reducible} (or FS-reducible). 
   
The following example shows how this lets us count the avoidance class $\Av(1423,2314)$.
\begin{eg}\label{1423_2314}
Note that a (moderately) quick calculation with Vatter's Maple package \texttt{WILFPLUS} reveals that 3214 and 4321 are both irreducible using traditional schemes. In fact, $k\dots21$ is ES-irreducible for all $k$, and no finite scheme exists.

However, 321 is FS-reducible, and thus so are its refinements 3214 and 4321. Suppose that $\mathbf{g}$ is a gap vector satisfying the condition [0,1,0,0], i.e. $\mathbf{g}$ has length 4 and $g_2 \ge 1$. It is not true that any permutation with prefix 321 and gap vector $\mathbf{g}$ must contain a forbidden pattern; for instance, $\mathbf{g} = [0,1,0,0]$ yields the permutation 3421. It is true, however, that as long as this gap condition is satisfied, $Z(\{1423, 2314\}, 321, [g_1, g_2, g_3, g_4]) = Z(\{1342, 3124\}, d_3(321), d_3([g_1, g_2, g_3+g_4]))$.

To see this, suppose that a permutation $\s$ has downfix 321 and gap vector $[g_1, g_2, g_3, g_4]$ (thus, $\s_{g_1+1} = 3, \s_{g_1+g_2 +2} =2,$ and $\s_{g_1+g_2+g_3+3} = 3$). By way of contradiction, suppose that $g_2 \ge 1$, but the 1 in $\s$ is not FS-reducible. In other words $\s$ contains either a 1423 or 2314 pattern when the 1 is present, but contains no forbidden pattern when it is removed. Suppose that the 1 participates in a 1423 pattern. When the 1 is removed, either the 2 or 3 can fill in for it in the 1423 pattern, so a forbidden pattern still occurs.

The more difficult case is if 1 participates in a 2314 pattern. Obviously, 1 serves as the 1 in this pattern. We consider three subcases based on which element of $\s$ serves as the 2. First, suppose that 2 serves as the 2. To complete the pattern, we find $\s_i, \s_j$ with $3 < \s_i < \s_j$ such that $ g_1+g_2+2 < i < g_1+g_2+g_3 + 3 < j$. Choose $k$ such that $g_1 + 1 < k < g_1+g_2 +2$ (we know this is possible since $g_2 \ge 1$). If $\s_k < \s_j$, then $3\s_k 2\s_j$ form a 2314 pattern, while if $\s_k > \s_j$, then $3\s_k \s_i \s_j$ for a 1423 pattern.

In the second case, some $\s_l$ with $l \le g_1+1$ serves as the 2. As before, complete the pattern, this time by finding $\s_i, \s_j$ such that  $l < i < g_1+g_2 + g_3 + 3 < j$. If $i < g_1+g_2+2$ (in other words, if $\s_i$ occurs before 2) then $\s_l\s_i2\s_j$ is a 2314 pattern. Otherwise, $ g_1+g_2+2 < i$, and we are back in case 1. 

In the third case, some $\s_l$ with $g_1+1 < l < g_1+g_2+g_3+3$ serves as the 2. Complete the pattern by finding $\s_i, \s_j$ with $l < \s_i < g_1+g_2+g_3+3 < \s_j$, and note that $3\s_i 1\s_j$ is also a 2314 pattern and so we are back in case 2.

We have now dealt with the difficult case when $g_2 \ge 1$, and we are left with the easier case when $g_2 = 0$. Let $\s$ be a permutation with a 321 downfix and gap vector $[g_1,0,g_3,g_4]$. Since $2$ and $3$ are consecutive both in terms of their position in $\s$ and their values, but neither 1423 nor 2314 contain decreasing elements which are consecutive in both senses, it follows that no occurrence of a forbidden pattern uses both $2$ and $3$. Further, 3 can be replaced in any forbidden pattern with 2 (or vice-versa) and so 3 is deletable.

\end{eg}

As the previous example shows, verifying one of these reducibility rules by hand is largely a matter of making simple arguments for many tedious special cases. We will see in Section \ref{find_scheme} how a computer can do all this work for us.

\subsection{Counting with Schemes} \label{details}

In this section, we describe flexible schemes from the perspective of a computer, beginning with the data structure in which they are stored and then showing that a scheme allows for polynomial-time enumeration of a permutation class.

A scheme is a collection of \textit{replacement rules} which allow any sufficiently long permutation downfix to be replaced by a shorter one. As much as possible, we model these rules after Zeilberger's \textit{VZ-triples} defined in \cite{VZSchemes}.
 
A replacement rule is a pair $[\pi, \mathbf{H}]$, where $\pi$ is a downfix and $\mathbf{H} = \big[[\mathbf{h_1}, r_1],\dots,[\mathbf{h_k},r_k]\big]$ is a list of pairs, each consisting of a gap condition and the index of an element which can be deleted if that condition is satisfied and all previous conditions are not satisfied \footnote{For performance reasons, this is not exactly how they are implemented in the Maple package \texttt{Flexible\_Scheme}, but it is a theoretically simpler description and is equivalent in that \texttt{Flexible\_Scheme} can find a reduction rule for a certain downfix if and only if that downfix is FS-reducible in this sense.}.

\begin{eg}
Suppose we are trying to avoid the pattern 123. We obtain the scheme $\bigg\{\Big[[1],\big[\big]\Big],\Big[[1,2],\big[[[0,0,1],0],[[0,0,0],2]\big]\Big],\Big[[2,1],\big[[[0,0,0],1]\big]\Big]\bigg\}.$ 

This scheme is interpreted as follows. If a permutation has the downfix 1, we do not yet know how to reduce it ($\mathbf{H}=[]$ is a convention indicating that no reduction is possible yet). Since we can't reduce it, we must consider both refinements of 1: 12 and 21.  If a permutation has the downfix 12, we check to see if it satisfies the gap condition [0,0,1], i.e. if it has at least one element following the 2. If so, the permutation cannot possibly avoid 123 (this is indicated by the $0$ following the gap condition). If not, we check to see if it satisfies the gap condition [0,0,0]; since every permutation always does, we find that the second element of the downfix is FS-reducible. Finally, if a permutation has the downfix 21, we check to see if it satisfies the gap condition [0,0,0]; again it must, and so the first element of the downfix is FS-reducible.
\end{eg}

We call the algorithm which follows these rules \texttt{FindTerm}. This algorithm inputs a scheme $S$, a downfix $\pi$, and a gap vector $\mathbf{g}$, and it outputs the number of elements in $Y(\pi, \mathbf{g})$ which are in the permutation class that $S$ describes. Pseudocode for \texttt{FindTerm} is given in Algorithm \ref{FT}.

\begin{algorithm}
\caption{\texttt{FindTerm}\label{FT}}
	\SetKwFunction{Find}{Find}
	\SetKwFunction{FindFirst}{FindFirst}
	\SetKwFunction{Break}{break}
	\SetKwFunction{Len}{length}
	\SetKwFunction{Return}{return}
	
	\Find $r \in S$ such that $r[1] = \pi$\\
	$\mathbf{H}:=r[2]$ \\
	\If{$\mathbf{H} = []$}{
		\If{ $\mathbf{g} = [0,0,\dots, 0]$}{
			\Return(1)
		}
		$output := 0$ \\
		\For{$i := 1$ \KwTo \Len$(\mathbf{g})+1$}{
			\For{$j := 0$ \KwTo $\mathbf{g}[i]-1$}{
				$output\text{ += }(\texttt{FindTerm}(S, f_i(\pi), f_{i,j}(\mathbf{g})))$
			}
		}
		\Return($output$)
	}
	\FindFirst $\mathbf{h} \in \mathbf{H}$ such that $\mathbf{h}[1] \preceq \mathbf{g}$\\
	$i := \mathbf{h}[2]$\\
	\If{i = 0}{
		\Return(0)
	}
	\Else{
		\Return$(\texttt{FindTerm}(S, d_i(\pi), d_i(\mathbf{g})))$
	}
\end{algorithm}

Given a scheme $S$ for a permutation class, we find the number of permutations of length $n$ which avoid it as $\texttt{FindTerm}(S, [], [n])$. We claim that the runtime is polynomial in $n$.

\begin{prop} Let $S$ be a scheme with depth $d$ and suppose that each rule has at most $t$ gap conditions. Then, $\texttt{FindTerm}(S,[],[n])$ runs in $O(n^{d+2})$ time. 
\end{prop}

\begin{proof}
Note that in every recursive call, $\mathbf{g}$ has no more than $d+1$ elements, each of which is in $[0,n-1]$. Similarly, $\pi$ has no more than $d$ elements. Based on these (extremely rough) bounds, we conclude that we need to call \texttt{FindTerm} recursively at most $d!n^{d+1}$ times. Within each call, if $\mathbf{H} = []$, we loop first over all $i \in [1, \texttt{length}(\mathbf{g})+1]$ ($\le d+2$ values), and then over all $j \in \mathbf{g}[i]$ ($\le n$ values). On the other hand, if $\mathbf{H} \neq []$, we need to look through $ \le t$ possible $\mathbf{h}$s (making $\le n$ comparisons each time) to find the first one with $\mathbf{h}[1] \preceq \mathbf{g}$, and then make a single recursive call.
\end{proof}

\subsection{Automatic Scheme Discovery} \label{find_scheme}

Of course, a scheme is only useful if we can find it in the first place. Given some downfix $\pi$, the idea is to test every possible gap condition $\mathbf{h}$ to see if it either guarantees a forbidden pattern or has some element $r$ which is FS-reducible for all gap vectors satisfying $\mathbf{h}$. Once we find such an $\mathbf{h}$, we know how to reduce $\pi$ whenever its gap vector satisfies $\mathbf{h}$, so for future $\mathbf{h'}$, we only need to find $r$ which is FS-reducible for all gap vectors satisfying $\mathbf{h'}$ and failing to satisfy $\mathbf{h}$ (or have $\mathbf{h'}$ guarantee a forbidden pattern). The following proposition, based on Proposition 6.2 from \cite{VSchemes}, lets a computer verify that $r$ is FS-reducible for all gap vectors satisfying $\mathbf{h}$ and not satisfying $\mathbf{h_1},\mathbf{h_2},\dots, \mathbf{h_k}$ by checking only a finite number of gap vectors. Note that $||B||_\infty$ is the length of the largest element of $B$ and $||\mathbf{g}||_1$ is the sum of the elements of $\mathbf{g}$.

\begin{prop}\label{check} Let $B$ be a set of forbidden patterns, $\pi$ be a downfix, and $\mathbf{h}, \mathbf{h_1}, \mathbf{h_2}, \dots, \mathbf{h_k}$ be gap conditions. Suppose that 
 \[|Z(B; \pi; \mathbf{g})| = |Z(B;d_r(\pi) ; d_r(\mathbf{g}))|\]
 for all $\mathbf{g}$ with $||\mathbf{g}||_1 \le ||B||_\infty -1 + ||\mathbf{h}||_1$ which satisfy $\mathbf{h}$ but fail to satisfy any of $\mathbf{h_1}, \mathbf{h_2}, \dots, \mathbf{h_k}$. Then the equality holds for all $\mathbf{g}$ which satisfy $\mathbf{h}$ but fail to satisfy any of $\mathbf{h_1}, \mathbf{h_2}, \dots, \mathbf{h_k}$.
 \end{prop}
 \begin{proof}
 It is clear that $|Z(B; \pi; \mathbf{g})| \le |Z(B;d_r(\pi) ; d_r(\mathbf{g}))|$ always holds since each permutation in $Z(B; \pi; \mathbf{g})$ can have the downfix element $r$ removed to yield a distinct permutation in $Z(B;d_r(\pi) ; d_r(\mathbf{g}))$.
 
To see that $ \ge$ also holds, fix $\mathbf{g}$ satsifying $\mathbf{h}$ but not $\mathbf{h_1}, \mathbf{h_2}, \dots, \mathbf{h_k}$. Our strategy is to show that if there is any $\s \not\in Z(B;\pi;\mathbf{g})$ such that removing the downfix element $r$ gives a permutation in $Z(B;d_r(\pi);d_r(\mathbf{g}))$, then there is such a $\s$ corresponding to a small gap vector.

Consider a permutation $\s\in Y(\pi, \mathbf{g})$ but not in $Z(B; \pi; \mathbf{g})$, and suppose that removing the $r$ in the downfix of $\s$ eliminates the forbidden pattern so that the resulting permutation is in $Z(B;d_r(\pi) ; d_r(\mathbf{g}))$. Choose an occurrence of a pattern $b \in B$ in $\s$, and suppose it uses the elements $\s_{i_1},\s_{i_2},\dots,\s_{i_k}$ (where $k \le ||B||_\infty$). Form $\s'$ by removing from $\s$ all the elements which do not participate in this pattern; since at least one element from the pattern ($r$) came from the downfix, we have at most $||B||_\infty-1$ elements not in the downfix.
 
Up until this point, the proof has been identical to the proof of Proposition 6.2 in \cite{VSchemes}, but now we need to ensure that $\s'$ has a gap vector $\mathbf{g'} \succeq \mathbf{h}$. This can be done by replacing the elements of $\s$ which ensured that $\mathbf{g} \succeq \mathbf{h}$; in the worst case we have removed all the elements of $\s$ which lay in gaps with positive minimum sizes, and so we need to replace $||\mathbf{h}||_1$ elements. Altogether, $\s'$ has $\le ||B||_\infty -1 + ||\mathbf{h}||_1$ elements following its prefix.

To complete the proof, we note that by construction $\s'$ has gap vector $\mathbf{g'}$ satisfying $||\mathbf{g'}||_1 \le ||B||_\infty -1 + ||\mathbf{h}||_1$. By construction, $\mathbf{g'}$ satisfies $\mathbf{h}$, and since $\mathbf{g'} \preceq \mathbf{g}$, it also fails to satisfy any of $\mathbf{h_1}, \mathbf{h_2}, \dots, \mathbf{h_k}$. Finally, when the $r$ in the prefix is removed, $\s'$ ceases to contain any pattern of $B$.
 \end{proof}
 
 We conclude this section with a caveat regarding the performance of flexible schemes. In the next section we will show that flexible schemes provide polynomial-time enumeration for any permutation class with a finite traditional scheme or a regular insertion encoding, and many other classes besides. However, there is no free lunch. Compared with regular insertion encodings, enumeration is much slower because flexible schemes simply do not recognize the underlying C-finite structure of the enumeration sequences they produce. As a result, instead of enumeration in linear time, we have to settle for enumeration in $O(n^d)$ time. Compared to traditional schemes, meanwhile, flexible schemes may require much longer gap conditions, and so may be more difficult to build. With traditional schemes, we need only consider gap conditions of size $\le ||B||_\infty -1$, and so it is reasonable to simply try all possible gap conditions. Here, we may need much longer gap conditions, and so we must impose an artificial limit on the ones we will consider. In practice, though, this seems to be only a minor disadvantage (see Section \ref{results}).
  
\section{Sufficient Conditions for Flexible Schemes}\label{conditions}
For some automatic enumeration algorithms, we know precisely when they will succeed. As noted earlier, permutation classes have regular insertion encodings if and only if they contain finitely many vertical alternations. Enumeration schemes have proven to be more difficult to analyze, however. For all but a handful of special cases, we can only conclude that a finite enumeration scheme exists when we find one, and we can only conjecture that one does not exist when we do a lot of work and still fail to find one. In this section, we prove that finite flexible schemes exist whenever a finite traditional scheme or regular insertion encoding does.
 
First, we show that every downfix which is ES-reducible is also FS-reducible. As defined in \cite{VSchemes}, a downfix is ES-reducible if and only if there exists $r$ such that for all $\mathbf{g}$ either $|Z(B, \pi, \mathbf{g})| = 0$ or $|Z(B, \pi, \mathbf{g})| = |Z(B, d_r(\pi), d_r(\mathbf{g}))|$.

A downfix is FS-reducible, meanwhile, if and only if there exists a finite list of gap-conditions $\mathbf{h_1}, \mathbf{h_2}, \dots, \mathbf{h_k}$ with $\mathbf{h_k}=[0,0,\dots,0]$ and a corresponding list of integers $r_1, r_2, \dots r_k$ (where each integer is in $\{0,1,\dots, |\pi|\}$) such that for all $\mathbf{g}$, if $i$ is chosen minimally so that $\mathbf{g} \succeq \mathbf{h_i}$, then
\[|Z(B,\pi, \mathbf{g})| = 
\begin{cases}
0 &\text{ if } r_i =0
\\
|Z(B,d_{r_i}(\pi), d_{r_i}(\mathbf{g}))| &\text{ otherwise}
\end{cases}
\]

We claim that if a downfix is ES-reducible, then that downfix is also FS-reducible. As noted in \cite{VSchemes}, the set $\{\mathbf{g}: |Z(B, \pi, \mathbf{g})| \neq 0\}$ is a lower order ideal in the lattice $\mathbf{N}^{|\pi|+1}$ and has a finite basis $\mathbf{h_1}, \mathbf{h_2}, \dots, \mathbf{h_k}$ such that $|Z(B, \pi, \mathbf{g})| = 0$ if and only if $\mathbf{g} \succeq\mathbf{h_i}$ for some $i$. Therefore, we can simply use these $\mathbf{h}$s along with $\mathbf{h_{k+1}} = [0,0,\dots,0]$ and $(r_i)_{i=1}^{k+1}$ where $r_i = 0$ for $i \le k$ and $r_{k+1} = r$ to fulfill the conditions for FS-reducibility.

Next, we show that a regular insertion encoding also guarantees a finite flexible scheme.

\begin{thm}
Let $B$ be a set of forbidden patterns, and suppose the class of permutations avoiding $B$ has a regular insertion encoding. Then, that class also has a finite flexible scheme.\end{thm}
\begin{proof}
Recall that the permutation classes with regular insertion encodings are exactly those with finitely many vertical alternations. For such a class, we can find $k$ such that no vertical alternation in the class is longer than $2k$, and so, for every downfix $\pi$, any gap vector $\mathbf{g}$ with $k+1$ positive entries is nonviable (because every $\s \in Y(\pi, \mathbf{g})$ contains as a subsequence a vertical alternation of length $2k+1$). 

Fix some sufficiently long $\pi$. We first show that every gap condition $\mathbf{h}$ consisting of $k$ 1 entries and $(\text{length}(\mathbf{h}) - k)$ 0 entries has an element $r$ which is FS-reducible for every $\mathbf{g}$ satisfying $\mathbf{h}$. Then, for every gap condition $\mathbf{h'}$ consisting of $(k-1)$ 1 entries and $(\text{length}(\mathbf{h}) - k +1)$ 0 entries, we find an $r$ which is FS-reducible for every $\mathbf{g}$ satisfying $\mathbf{h'}$ and not satisfying any of the earlier $\mathbf{h}$ with $k$ 1 entries. We continue in this way until we have found an $r$ which is FS-reducible for every $\mathbf{g}$ satisfying $[0,0,\dots,0]$ but not satisfying any $\mathbf{h}$ which contains a 1. Every $\mathbf{g}$ with $j$ positive entries satisfies an $\mathbf{h}$ with $j$ 1 entries but not any $\mathbf{h}$ with more than $j$ 1 entries, and so this will show that $\pi$ is FS-reducible.

Let $\pi$ be a downfix and $\mathbf{h}$ be a gap condition with $k$ ones. Consider some element $i$ of $\pi$, and suppose it is not FS-reducible. Then, there exists some $\s$ with downfix $\pi$ and gap vector $\mathbf{g} \succeq \mathbf{h}$ which contains a pattern of $B$, but which does not contain such a pattern when $i$ is removed. Furthermore, by Proposition \ref{check} we can assume that $||\mathbf{g}||_1 \le ||B||_\infty - 1 + k$. Following Vatter in \cite{VIns}, we say that $\s$ witnesses $i$. Note that $i$ is present in every occurrence in $\s$ of every pattern of $B$, and so $\s$ can witness at most $||B||_\infty$ elements $i$.

Now, we just need to show that only finitely many $\s$s can witness elements. Since $\mathbf{g} \succeq \mathbf{h}$, all entries of $\mathbf{g}$ are 0 except for the $k$ which $\mathbf{h}$ forced to be positive. Thus, there are finitely many possible $\mathbf{g}$s which can provide witnesses. For some fixed $\mathbf{g}$ there are $||\mathbf{g}||_1!$ possible permutations $\s$ with downfix $\pi$ and gap vector $\mathbf{g}$. Therefore, only a fixed number of downfix elements can be witnessed, so, if $\pi$ is large enough, there is a downfix element which is not witnessed and hence is FS-reducible.

At this point, we have shown that for all sufficiently long $\pi$ and gap conditions $\mathbf{h}$ with $k$ ones, there exists an element $r$ which is FS-reducible for every viable $\mathbf{g} \succeq \mathbf{h}$. The proof is essentially the same for any $j < k$, but we will write it out anyway for the sake of completeness. 

Let $\pi$ be a downfix and $\mathbf{h}$ be a gap condition with $j$ ones for some $0 \le j < k$. Consider some element $i$ of $\pi$, and suppose it is not FS-reducible. Then, there exists some $\s$ which contains a pattern of $B$, but which does not contain such a pattern when $i$ is removed. This $\s$ has downfix $\pi$ and gap vector $\mathbf{g}$ satisfying $\mathbf{h}$ but not satisfying $\mathbf{h'}$ for any $\mathbf{h'}$ with more than $j$ positive entries. Again Proposition \ref{check} lets us assume that $||\mathbf{g}||_1 \le ||B||_\infty - 1 + j$. Also as before, $\s$ can witness at most $||B||_\infty$ elements $i$.

We know that all entries of $\mathbf{g}$ are 0 except for the $j$ which $\mathbf{h}$ forced to be positive (otherwise, $\mathbf{g}$ would satisfy some other $\mathbf{h'}$ with more 1s). Again, this means that only finitely many $\mathbf{g}$s can provide witnesses, and each $\mathbf{g}$ provides at most $||\mathbf{g}||_1!$ witnesses. Thus, if $\pi$ is large enough, there is a downfix element which is FS-reducible. As noted at the end of the second paragraph of this proof, that is enough to show that $\pi$ is FS-reducible. Since this is true for all sufficiently large $\pi$, a finite flexible scheme exists.
\end{proof}

\section{Empirical Results}\label{results}
We tried to find regular insertion encodings, traditional schemes, and flexible schemes for several different permutation classes. In particular we looked at the avoidance classes of pattern sets $B$ where $B$ consisted of either a single pattern of length 3, 4, or 5, a pair of patterns of length 3, a pair of patterns of length 4, or a pattern of length 4 and another of length 5.

For each of these possible pattern lengths, Table \ref{ER} shows how many classes have regular insertion encodings, how many have finite traditional schemes, how many have finite flexible schemes, and how many of those with finite flexible schemes did not have either a regular insertion encoding or a finite traditional scheme. For finding regular insertion encodings, we used Vatter's package \texttt{InsEnc} from \cite{VIns}, for finding traditional schemes we used Zeilberger's package \texttt{VATTER} from \cite{VZSchemes} and for finding flexible schemes we used our own \texttt{Flexible\_Scheme}. In about half of the cases, we also checked for traditional schemes using \texttt{WILFPLUS} from \cite{VSchemes}, and we plan to test the other half as well.

While we can safely conclude that we found regular insertion encodings whenever they exist, the same is not true of schemes. When finding traditional schemes using \texttt{WILFPLUS}, we only considered downfixes of length $\le 8$; it is conceivable that finite schemes exist, but simply require longer downfixes, and so we did not find them. When finding traditional schemes using \texttt{VATTER} and flexible schemes using \texttt{Flexible\_Scheme}, we only considered downfixes of length $\le 8$ and gap conditions with $l_1$ norm $\le 2$. In addition, if we could not determine whether an avoidance class had a finite scheme after 36 hours of computation, we recorded it as not having a scheme.

Because of these constraints, it is conceivable that either \texttt{InsEnc} or \texttt{WILFPLUS} could have found a regular insertion encoding or a finite scheme for some class that \texttt{Flexible\_Scheme} failed to find one for. In practice, this happened for three of the avoidance classes we considered (all of which were avoiding a length 4 and length 5 pattern). 

\begin{table}[h]\caption{Empirical Results}\label{ER}
\begin{tabular}{| c | c | c | c | c | c |}
\hline
Pat length\tablefootnote{$[n], [m]$ refers to classes avoiding one pattern of length $n$ and one of length $m$, $[n]$ refers to patterns avoiding one pattern of length $n$} & Sym Classes\tablefootnote{Every permutation class has up to 8 symmetries given by inverting, reversing, and complementing its patterns. Every class has the same enumeration sequence as its symmetries, so we really only care about the number of symmetry classes which can be enumerated, not the total number of permutation classes} & Ins. Enc. & ES & FS & New with FS \\ 
\hline
[3] & 2 & 0 & 2 & 2 & 0 \\
\hline
[4] & 7 & 0 & 2 & 2& 0\\
\hline
[5] & 23 & 0 & 2 & 2 & 0\\
\hline
[3], [3] & 5 & 5&5 &5 &0 \\
\hline
[4], [4] & 56 & 13 & 33 & 44 & 9 \\
\hline
[4], [5] & 434 & 30 & 112 & 173 & 59  \\
\hline
\end{tabular}
\end{table}

The permutation classes avoiding two length four permutations have been particularly well studied by previous authors. Generating functions are known for all but three of them, and these are conjectured not to have any D-finite generating functions in \cite{no_gf}. Two of these (avoiding \{4321, 4231\} and \{4312, 4123\}) have finite traditional schemes, and the remaining one (avoiding \{4231, 4123\}) has a finite flexible scheme. As indicated in Table \ref{ER}, 12 of the 56 symmetry classes lack finite flexible schemes as far as we can tell. These are the classes avoiding $\{1234, 3412\}$, $\{1324, 2143\}$, $\{1324, 3412\}$, $\{1324, 2341\}$, $\{1324, 4231\}$, $\{1324, 2413\}$, $\{1324, 2431\}$, $\{1342, 1423\}$, $\{1342, 2413\}$, $\{1432,2413\}$, $\{2143, 2413\},$ and $\{2413, 3142\}$. The interested reader can find links to the enumeration sequences and generating functions of all of these avoidance classes at \cite{wiki2x4}.

\section{Maple Package}
This paper is accompanied by the package \texttt{Flexible\_Scheme} available at the author's \href{http://math.rutgers.edu/~yb165/Flexible_Scheme/Flexible_Scheme.txt}{website}. This package has two main procedures. The first, HasScheme, inputs a set of permutations, a maximum depth of downfixes to consider, and a maximum $l_1$ norm of gap conditions to consider and outputs a scheme if one exists within the given constraints. The second function, SeqS, inputs a scheme and integer $n$, and outputs the first $n$ terms of the enumeration sequence given by that scheme. For more details, load the package in Maple and call the function Help with either of the two functions as the argument.

\section{Acknowledgments}
The author is grateful to his advisor Doron Zeilberger for his support and introduction to enumeration schemes. He would also like to thank Vince Vatter for introducing him to the insertion encoding and suggesting the problem of unifying the two techniques, as well as for several helpful comments on this paper.

 \bibliography{Generic_Bib}
 \bibliographystyle{abbrv}
 
 \end{document}